\documentclass[12pt]{amsart}
\usepackage{amscd,amssymb,amsmath,multicol,float,scalefnt}
\restylefloat{table}
\newtheorem{thm}[equation]{Theorem}
\numberwithin{equation}{section}
\newtheorem{cor}[equation]{Corollary}
\newtheorem{expl}[equation]{Example}
\newtheorem{rmk}[equation]{Remark}

\newtheorem{lem}[equation]{Lemma}

\newtheorem{prop}[equation]{Proposition}

\newtheorem{tab}[equation]{Table}

\begin{document}
\raggedbottom \voffset=-.7truein \hoffset=0truein \vsize=8truein
\hsize=6truein \textheight=8truein \textwidth=6truein
\baselineskip=18truept

\def\mapright#1{\ \smash{\mathop{\longrightarrow}\limits^{#1}}\ }
\def\mapleft#1{\smash{\mathop{\longleftarrow}\limits^{#1}}}
\def\mapup#1{\Big\uparrow\rlap{$\vcenter {\hbox {$#1$}}$}}
\def\mapdown#1{\Big\downarrow\rlap{$\vcenter {\hbox {$\ssize{#1}$}}$}}
\def\mapne#1{\nearrow\rlap{$\vcenter {\hbox {$#1$}}$}}
\def\mapse#1{\searrow\rlap{$\vcenter {\hbox {$\ssize{#1}$}}$}}
\def\mapr#1{\smash{\mathop{\rightarrow}\limits^{#1}}}
\def\ss{\smallskip}
\def\vp{v_1^{-1}\pi}
\def\at{{\widetilde\alpha}}
\def\sm{\wedge}
\def\la{\langle}
\def\ra{\rangle}
\def\on{\operatorname}
\def\squarecup{\sqcup}
\def\spin{\on{Spin}}
\def\kbar{{\overline k}}
\def\qed{\quad\rule{8pt}{8pt}\bigskip}
\def\ssize{\scriptstyle}
\def\a{\alpha}
\def\bz{{\Bbb Z}}
\def\im{\on{im}}
\def\ct{\widetilde{C}}
\def\ext{\on{Ext}}
\def\sq{\on{Sq}}
\def\eps{\epsilon}
\def\ar#1{\stackrel {#1}{\rightarrow}}
\def\br{{\bold R}}
\def\bC{{\bold C}}
\def\bA{{\bold A}}
\def\bB{{\bold B}}
\def\bD{{\bold D}}
\def\bh{{\bold H}}
\def\bQ{{\bold Q}}
\def\bP{{\bold P}}
\def\bx{{\bold x}}
\def\bo{{\bold{bo}}}
\def\si{\sigma}
\def\ol#1{\overline{#1}}
\def\Ebar{{\overline E}}
\def\dbar{{\overline d}}
\def\Sum{\sum}
\def\tfrac{\textstyle\frac}
\def\tb{\textstyle\binom}
\def\Si{\Sigma}
\def\w{\wedge}
\def\equ{\begin{equation}}
\def\b{\beta}
\def\G{\Gamma}
\def\g{\gamma}
\def\k{\kappa}
\def\psit{\widetilde{\Psi}}
\def\tht{\widetilde{\Theta}}
\def\psiu{{\underline{\Psi}}}
\def\thu{{\underline{\Theta}}}
\def\aee{A_{\text{ee}}}
\def\aeo{A_{\text{eo}}}
\def\aoo{A_{\text{oo}}}
\def\aoe{A_{\text{oe}}}
\def\fbar{{\overline f}}
\def\endeq{\end{equation}}
\def\sn{S^{2n+1}}
\def\zp{\bold Z_p}
\def\A{{\cal A}}
\def\P{{\mathcal P}}
\def\cj{{\cal J}}
\def\zt{{\bold Z}_2}
\def\bs{{\bold s}}
\def\bof{{\bold f}}
\def\bq{{\bold Q}}
\def\be{{\bold e}}
\def\Tor{\on{Tor}}
\def\ker{\on{ker}}
\def\coker{\on{coker}}
\def\da{\downarrow}
\def\colim{\operatornamewithlimits{colim}}
\def\zphat{\bz_2^\wedge}
\def\io{\iota}
\def\Om{\Omega}
\def\Prod{\prod}
\def\e{{\cal E}}
\def\exp{\on{exp}}
\def\abar{{\overline a}}
\def\xbar{{\overline x}}
\def\ybar{{\overline y}}
\def\zbar{{\overline z}}
\def\Vbar{{\overline V}}
\def\nbar{{\overline n}}
\def\bbar{{\overline b}}
\def\et{{\widetilde E}}
\def\ni{\noindent}
\def\coef{\on{coef}}
\def\zcl{\on{zcl}}
\def\den{\on{den}}
\def\lcm{\on{l.c.m.}}
\def\vi{v_1^{-1}}
\def\ot{\otimes}
\def\psibar{{\overline\psi}}
\def\mhat{{\hat m}}
\def\exc{\on{exc}}
\def\ms{\medskip}
\def\ehat{{\hat e}}
\def\etao{{\eta_{\text{od}}}}
\def\etae{{\eta_{\text{ev}}}}
\def\dirlim{\operatornamewithlimits{dirlim}}
\def\gt{\widetilde{L}}
\def\lt{\widetilde{\lambda}}
\def\st{\widetilde{s}}
\def\ft{\widetilde{f}}
\def\sgd{\on{sgd}}
\def\lfl{\lfloor}
\def\rfl{\rfloor}
\def\ord{\on{ord}}
\def\gd{{\on{gd}}}
\def\rk{{{\on{rk}}_2}}
\def\nbar{{\overline{n}}}
\def\lg{{\on{lg}}}
\def\cR{\mathcal{R}}
\def\cT{\mathcal{T}}
\def\N{{\Bbb N}}
\def\Z{{\Bbb Z}}
\def\Q{{\Bbb Q}}
\def\R{{\Bbb R}}
\def\C{{\Bbb C}}
\def\l{\left}
\def\r{\right}
\def\mo{\on{mod}}
\def\vexp{v_1^{-1}\exp}
\def\notimm{\not\subseteq}
\def\Remark{\noindent{\it  Remark}}

\def\*#1{\mathbf{#1}}
\def\0{$\*0$}
\def\1{$\*1$}
\def\22{$(\*2,\*2)$}
\def\33{$(\*3,\*3)$}
\def\ss{\smallskip}
\def\ssum{\sum\limits}
\def\dsum{\displaystyle\sum}
\def\la{\langle}
\def\ra{\rangle}
\def\on{\operatorname}
\def\o{\on{o}}
\def\U{\on{U}}
\def\lg{\on{lg}}
\def\a{\alpha}
\def\bz{{\Bbb Z}}
\def\eps{\varepsilon}
\def\br{{\bold R}}
\def\bc{{\bold C}}
\def\bN{{\bold N}}
\def\nut{\widetilde{\nu}}
\def\tfrac{\textstyle\frac}
\def\b{\beta}
\def\G{\Gamma}
\def\g{\gamma}
\def\zt{{\bold Z}_2}
\def\zth{{\bold Z}_2^\wedge}
\def\bs{{\bold s}}
\def\bx{{\bold x}}
\def\bof{{\bold f}}
\def\bq{{\bold Q}}
\def\be{{\bold e}}
\def\lline{\rule{.6in}{.6pt}}
\def\xb{{\overline x}}
\def\xbar{{\overline x}}
\def\Wbar{{\overline W}}
\def\ybar{{\overline y}}
\def\zbar{{\overline z}}
\def\ebar{{\overline \be}}
\def\nbar{{\overline n}}
\def\rbar{{\overline r}}
\def\Mbar{{\overline M}}
\def\et{{\widetilde e}}
\def\ni{\noindent}
\def\ms{\medskip}
\def\ehat{{\hat e}}
\def\xhat{{\widehat x}}
\def\lbar{\ell}
\def\minp{\min\nolimits'}
\def\N{{\Bbb N}}
\def\Z{{\Bbb Z}}
\def\Q{{\Bbb Q}}
\def\R{{\Bbb R}}
\def\C{{\Bbb C}}
\def\el{\ell}
\def\bl{\ell}
\def\TC{\on{TC}}
\def\dstyle{\displaystyle}
\def\ds{\dstyle}
\def\Remark{\noindent{\it  Remark}}
\title[On the zero-divisor-cup-length]
{On the zero-divisor-cup-length of spaces of oriented isometry classes of planar polygons}
\author{Donald M. Davis}
\address{Department of Mathematics, Lehigh University\\Bethlehem, PA 18015, USA}
\email{dmd1@lehigh.edu}
\date{March 17, 2016}


\maketitle
\begin{abstract} Using information about the rational cohomology ring of the space $M(\ell_1,\ldots,\ell_n)$ of oriented isometry classes of planar $n$-gons with the specified side lengths, we obtain bounds for the zero-divisor-cup-length ($\zcl$) of these spaces, which provide lower bounds for their topological complexity (TC). In many cases our result about the cohomology ring is complete and  we determine the precise $\zcl$. We find that there will usually be a significant gap between the bounds for $\TC$ implied by zcl and dimensional considerations.
 \end{abstract}

\section{Introduction}\label{intro}
The topological complexity, $\TC(X)$, of a topological space $X$ is, roughly, the number of rules required to specify how to move between any two points of $X$. A ``rule'' must be such that the choice of path varies continuously with the choice of endpoints. (See \cite[\S4]{F}.) Information about the cohomology ring of $X$ can be used to give a lower bound for $\TC(X)$.

Let $\bl=(\ell_1,\ldots,\ell_n)$ be an $n$-tuple of positive real numbers. Let $M(\bl)$ denote the space of oriented $n$-gons in the plane with successive side lengths $\ell_1,\ldots,\ell_n$, where polygons are identified under translation and rotation. Thus
$$M(\bl)=\{(z_1,\ldots,z_n)\in (S^1)^n:\sum \ell_iz_i=0\}/SO(2).$$
If we think of the sides of the polygon as linked arms of a robot, then $\TC(M(\bl))$ is the number of rules required to program the robot to move between any two configurations.

Let $[n]=\{1,\ldots,n\}$ throughout. We say that $\bl$ is {\it generic} if there is no subset $S\subset[n]$ for which $\ds\sum_{i\in S}\ell_i=\ds\sum_{i\not\in S}\ell_i$.
For such $\bl$,
$M(\bl)$ is an orientable $(n-3)$-manifold (\cite[Thm 1.3]{F}) and hence, by \cite[Cor 4.15]{F}, satisfies \begin{equation}\label{TCM}\TC(M(\bl))\le 2n-5.\end{equation}
A lower bound for topological complexity is obtained using
 the  zero-divisor-cup-length of $X$, $\zcl(X)$, which is the maximum number of elements $\a_i\in H^*(X\times X)$ satisfying $m(\a_i)=0$ and $\ds\prod_i\a_i\ne0$. Here $m:H^*(X)\otimes H^*(X)\to H^*(X)$ denotes the cup product pairing with rational coefficients, and $\a_i$ is called a zero divisor. Throughout the paper, all cohomology groups have coefficients in $\Q$, unless specified to the contrary. In \cite[Thm 7]{F2}, it was shown that \begin{equation}\label{zclz}\TC(X)\ge\zcl(X)+1.\end{equation}

In this paper, we obtain some new information about the rational cohomology ring $H^*(M(\bl))$ when $\bl$ is generic to obtain lower bounds for $\zcl(M(\bl))$ and hence for $\TC(M(\bl))$. Frequently, our description of the cohomology ring is complete (64 out of 134 cases when $n=7$), and we can give the best lower bound implied by ordinary cohomological methods. However, unlike the situation for isometry classes of polygons, i.e., when polygons are also identified under reflection, this lower bound is usually significantly less than $2n-5$.

Indeed, for the space of isometry classes of planar polygons,
$$\Mbar(\bl)=\{(z_1,\ldots,z_n)\in (S^1)^n:\sum \ell_iz_i=0\}/O(2),$$
the mod-2 cohomology ring was completely determined in \cite{HK}, and in \cite{D1} and \cite{D2} we showed that for several large families of $\bl$,
$$2n-6\le\TC(\Mbar(\bl))\le 2n-5,$$
the latter because $\Mbar(\bl))$ is also an $(n-3)$-manifold when $\bl$ is generic.
Note that for motions in the plane, $M(\bl)$ would seem to be a more relevant space than $\Mbar(\bl)$.
For the spaces $M(\bl)$ considered here, rational cohomology often gives slightly stronger bounds than does mod-2 cohomology.

 In Section \ref{EMSSsec}, we describe what we can say about the rational cohomology ring $H^*(M(\bl))$.
 In Section \ref{zclsec}, we  obtain information about $\zcl(M(\bl))$ and hence $\TC(M(\bl))$. Theorems \ref{upper} and \ref{lower} give upper and lower bounds for $\zcl(M(\bl))$.  See Table \ref{T1} for a tabulation when $n=8$.
In Section \ref{sec4}, we give an example, due to the referee, in which there are what we call ``exotic products'' in the cohomology ring. The possibility of these prevents us from making stronger zcl estimates.

We thank the referee for pointing out a mistake in an earlier version, and for pointing out a number of illustrative examples. We also thank Nitu Kitchloo for some early suggestions.

\section{The rational cohomology ring $H^*(M(\bl))$}\label{EMSSsec}
We assume throughout that $\ell_1\le\cdots\le\ell_n$.
It is well-understood (\cite[Prop 2.2]{HK}) that the homeomorphism type of $M(\bl)$  is determined by which subsets $S$ of $[n]$ are {\it short}, which means that
$\dsum_{i\in S}\ell_i<\tfrac12\dsum_{i=1}^n\ell_i$. For generic $\lbar$, a subset which is not short is called {\it long}.

Define a partial order on the power set of $[n]$  by $S\le T$ if  $S=\{s_1,\ldots,s_\ell\}$ and $T\supset\{t_1,\ldots,t_\ell\}$ with $s_i\le t_i$ for all $i$.
This order will be used throughout the paper, applied also to multisets.
As introduced in \cite{H}, the {\it genetic code} of $\lbar$  is the  set of maximal elements (called {\it genes}) in the set of short subsets of $[n]$ which contain $n$.
The homeomorphism type of $M(\bl)$  is determined by the genetic code of $\lbar$. Note that if $\lbar=(\ell_1,\ldots,\ell_n)$, then all genes have largest element $n$. We introduce the new terminology that if $\{n,i_r,\ldots,i_1\}$ is a gene, then $\{i_r,\ldots,i_1\}$ is called a {\it gee}. (Gene without the $n$.)
 We define a {\it subgee} to be a set of positive integers which is $\le$ a gee under the above ordering.

 The following result was proved in \cite[Thm 6]{FHS}.
 \begin{thm}\label{FHSthm} The rational cohomology ring $H^*(M(\bl))$ contains a subalgebra generated by classes $V_1,\ldots,V_{n-1}\in H^1(M(\bl))$
 whose only relations are that if $S=\{s_1,\ldots,s_k\}$ with $s_1<\cdots<s_k$, then $V_S:=V_{s_1}\cdots V_{s_k}$ satisfies $V_S=0$ iff $S$ is not a subgee of $\bl$.\end{thm}
 \ni In other words, the nonzero monomials in the $V_i$'s correspond exactly to the subgees. Of course, $V_i^2=0$, since $\dim(V_i)$ is odd.

It is well-known (e.g. \cite[Expl 2.3]{Geom}) that if the genetic code of $\bl$ is $\la\{n,n-3,n-4,\ldots,1\}\ra$, then $M(\bl)$ is homeomorphic to $(S^1)^{n-3}\sqcup (S^1)^{n-3}$. We will exclude this case from our analysis and use the following known  result, in which, as always, $\bl=(\ell_1,\ldots,\ell_n)$.
\begin{prop}\label{notmax} $($\cite[Rmk 2.8]{Geom}$)$ If the genetic code of $\bl$ does not equal $\la\{n,n-3,\ldots,1\}\ra$, then all genes have cardinality less than $n-2$, and $M(\bl)$ is a connected $(n-3)$-manifold.\end{prop}

From now on, let $m=n-3$ denote the dimension of  $M(\bl)$, and let $W_\emptyset$ denote the orientation class of $H^m(M(\bl))$.
We obtain
\begin{thm} A basis for $H^*(M(\lbar))$ consists of the classes $V_S$ of Theorem \ref{FHSthm} such that $S$ is a subgee of $\bl$,  together with classes $W_S\in H^{m-|S|}(M(\bl))$, for exactly the same $S$'s, satisfying that  $$ V_SW_{S'}=\delta_{S,S'}W_\emptyset\quad\text{if } |S'|=|S|.$$
Also $V_SV_{S'}=V_{S\cup S'}$ if $S$ and $S'$ are disjoint and $S\cup S'$ is a subgee of $\bl$, while $V_SV_{S'}=0$ otherwise. Finally,
$W_SW_{S'}=0$ whenever $|W_S|+|W_{S'}|=m$.\label{basis}\end{thm}
\begin{proof} By \cite[Thm 1.7]{F}, for all $i$, the $i$th Betti number of $M(\bl)$ equals the number of $V_S$'s described in Theorem \ref{FHSthm} of degree $i$ plus the number of such $V_S$'s of degree $m-i$.
By Theorem \ref{FHSthm}, our classes $V_S$ are linearly independent in $H^*(M(\bl))$ and all products $V_SV_{S'}$ are zero except those listed in our set. By Lemma \ref{lemma}, the nonsingularity of the Poincar\'e duality pairing implies that there are classes $W_S$ which pair with the classes $V_S$ and with each other in the claimed manner, and the Betti number result implies that there are no additional classes.
\end{proof}

The following elementary lemma was used in the preceding proof.
This lemma is applied to $U=H^i(M(\bl))$, $U'=H^{m-i}(M(\bl))$, $\{u_1,\ldots,u_k\}$ the set of $V_S$'s in $H^i(M(\bl))$, and $\{u'_{k+1},\ldots,u'_t\}$ the set of $V_S$'s in  $H^{m-i}(M(\bl))$.

\begin{lem} \label{lemma}Suppose $U$ and $U'$ are $t$-dimensional vector spaces over $\Q$ and $\phi:U\times U'\to\Q$ is a nonsingular bilinear pairing. Suppose $\{u_1,\ldots,u_k\}\subset U$ is linearly independent, as is $\{u'_{k+1},\ldots,u'_t\}\subset U'$, and $\phi(u_i,u'_j)=0$ for $1\le i\le k<j\le t$. Then there exist bases $\{u_1,\ldots,u_t\}$ and $\{u'_1,\ldots,u'_t\}$ of $U$ and $U'$ extending the given linearly-independent sets and satisfying $\phi(u_i,u'_j)=\delta_{i,j}$.\end{lem}
\begin{proof} For $1\le i\le k$, let $\psi_i:U\to\Q$ be any homomorphism for which $\psi_i(u_j)=\delta_{i,j}$ for $1\le j\le k$. By nonsingularity, there is $u'_i\in U'$ such that $\phi(u,u'_i)=\psi_i(u)$ for all $u\in U$. To see that $\{u'_1,\ldots,u'_t\}$ is linearly independent, assume $\sum c_\ell u'_\ell=0$. Applying $\phi(u_i,-)$ implies that $c_i=0$, $1\le i\le k$,
while linear independence of $\{u'_{k+1},\ldots,u'_t\}$ then implies that $c_{k+1}=\cdots=c_t=0$. Nonsingularity now implies that there are classes $u_i$ for $i>k$ such that $\phi(u_i,u'_j)=\delta_{i,j}$ for all $j$, and linear independence of the $u_i$'s is immediate.
\end{proof}

Results similar to Theorem \ref{basis} and Proposition \ref{VW} below were also presented, in slightly different situations, in \cite[Rmk 10.3.20]{book} and \cite[Prop A.2.4]{Fromm}.

Let $s$ denote the size of the largest gee of $\bl$. The only $V_S$'s occur in gradings $\le s$, and so the only $W_S$'s occur in grading $\ge m-s$. If $2(m-s)\le m-1$ (i.e., $m\le 2s-1$),
then there can be nontrivial products of $W_S$'s, about which we apparently have little control.

The following simple result gives excellent information about products of $V$ classes times $W$ classes.
In particular, if $m\ge2s$, the entire ring structure is determined! See Corollary \ref{nice}. When $m=4$, this is the case for 64 of the 134 equivalence classes of $\bl$'s, as listed in \cite{web}.

\begin{prop} \label{VW}  Let $\rho_i(T)$ denote the number of elements of $T$ which are greater than $i$. Modulo polynomials in $V_1,\ldots,V_{n-1}$,
\begin{equation}\label{equiv}V_iW_S\equiv\begin{cases}(-1)^{\rho_i(T)}W_T&\text{if }S=T\sqcup\{i\}\\ 0&\text{if }i\not\in S.\end{cases}\end{equation}
In particular, if $s$ is the maximal size of gees and $m-|S|\ge s$, then
\begin{equation}\label{VSeq}V_iW_S=\begin{cases}(-1)^{\rho_i(T)}W_T&\text{if }S=T\sqcup\{i\}\\ 0&\text{if }i\not\in S.\end{cases}\end{equation}
\end{prop}
\begin{proof} Write $V_iW_S=\sum\a_PV_P+\sum\a'_QW_Q$ with $\a_P,\a'_Q\in\Q$. If $V_T$ is any monomial in grading $|S|-1$, then
$$(-1)^{\rho_i(T)}\delta_{S,T\cup\{i\}}W_\emptyset=V_TV_iW_S=\sum_Q\a'_Q\delta_{T,Q}W_\emptyset=\a'_TW_\emptyset,$$
as all monomials in the $V$'s are 0 in grading $m$. The first result follows immediately.

The second part follows since $|V_iW_S|=m-|S|+1$ and all polynomials in the $V$'s are 0 in grading $>s$.\end{proof}

\begin{cor} \label{nice} If $m\ge 2s$, where $s$ is the maximal gee size, then the complete  structure of the algebra $H^*(M(\ell))$ is given by Theorem \ref{FHSthm} and (\ref{VSeq}).
\end{cor}

We offer the following illustrative example, in which we have complete information about the product structure. Here we begin using the notation introduced in \cite{H} of writing genes (and gees) which are sets of 1-digit numbers by just concatenating those digits.
\begin{expl} Suppose the genetic code of $\bl$ is $\la 9421, 95\ra$. Then a basis for $H^*(M(\bl))$ is:
\begin{eqnarray*}0&&1\\
1&&V_1,V_2,V_3,V_4,V_5\\
2&&V_1V_2,V_1V_3,V_1V_4,V_2V_3,V_2V_4\\
3&&V_1V_2V_3,V_1V_2V_4,W_{123},W_{124}\\
4&&W_{12},W_{13},W_{14},W_{23},W_{24}\\
5&&W_1,W_2,W_3,W_4,W_5\\
6&&W_{\emptyset}.\end{eqnarray*}
The only nontrivial products of $V$'s are those indicated. All products of $W$'s are 0. The multiplication of $V_i$ by $W_S$ is given, up to sign, by removal of the subscript $i$, if $i\in S$, else 0.
\end{expl}

In the above example, $m=6$ and $s=3$. It is quite possible that a similarly nice product structure might hold in various cases in which $m<2s$. When it does not, we refer to nonzero products of $W$'s
or cases in which (\ref{VSeq}) does not hold as {\it exotic products}. In Section \ref{sec4}, we present an example, due to the referee, in which nontrivial exotic products occur.

One important class of examples in which $s\le 2m$ and so Corollary \ref{nice} applies is the space $M_{2k+1}$ of equilateral $(2k+1)$-gons. Here $\ell=(1,\ldots,1)$ and the genetic code is $\la\{2k+1,2k,\ldots,k+2\}\ra$, and so $s=k-1$ and $m=2k-2$.

\section{Zero-divisor-cup-length}\label{zclsec}
In this section we study the zero-divisor-cup-length $\zcl(M(\bl))$, where $\bl=(\ell_1,\ldots,\ell_n)$, $\bl$ is  generic, and its genetic code does not equal $\la\{n,n-3,\ldots,1\}\ra$. We also discuss the implications for topological complexity.

 Our first result is an upper bound, which will sometimes be sharp. See Table \ref{T1} for a tabulation when $n=8$. Recall that $m=n-3$.
\begin{thm}\label{upper} If $s$ is the largest cardinality of the gees of $\bl$, then $\zcl(M(\bl))\le 2s+2$.\end{thm}
\begin{proof}  For $u\in H^*(M(\bl))$, let $\overline u=u\ot1-1\ot u$. We first consider products of the form $\prod\overline{u_i}$.
A product of $a$ $\ol{V_i}$'s and $b$ $\ol{W_S}$'s has grading $\ge a+b(m-s)$. If $a>2s$, then $\prod\ol{V_i}=0$, so we may assume that $a\le 2s$. If $a+b\ge 2s+3$, then $b\ge3$ and
$$a+b(m-s)\ge 2s+3-b+b(m-s)=bm+1-(b-2)(s+1)\ge bm+1-(b-2)m=2m+1,$$
and so the product must be 0. We have used that $s\le m-1$ by Proposition \ref{notmax}.

Now we consider the possibility of more general zero divisors.
Let $\a_j$ denote a zero divisor which contains a term $A\ot B$ in which the total number of $V$-factors (resp.~$W$-factors) in $AB$ is $p_j$ (resp.~$q_j$) with $p_j+q_j\ge2$. Its grading is $\ge p_j+q_j(m-s)$. A product of $a$ $\ol{V_i}$'s, $b$ $\ol{W_S}$'s, and $c$ $\a_j$'s, with $a+b+c\ge 2s+3$ will be 0 if $a+\sum p_j>2s$, so we may assume $a+\sum p_j\le 2s$. This product, with $c\ge1$, has grading
\begin{eqnarray*}&\ge&a+b(m-s)+\sum p_j+(m-s)\sum q_j\\
&\ge&a+b(m-s)+\sum p_j+(m-s)(2c-\sum p_j)\\
&\ge&a+(b+2c)(m-s)+(m-s-1)(a-2s)\\
&=&(m-s-1)(a+b+2c-2s)+a+b+2c\\
&\ge&(3+c)(m-s-1)+2s+3+c\\
&=&2m+(c+1)(m-s)\\
&>&2m,\end{eqnarray*}
and hence is 0.
\end{proof}

Next we give our best result for lower bounds.  Recall that the partial order described just before Theorem \ref{FHSthm} is applied also to multisets.

\begin{center}
\begin{minipage}{6.05in}
\begin{thm}\label{lower}
\ \begin{enumerate}
\item[a.] If $G$ and $G'$ are gees of $\bl$, not necessarily distinct, and there is an inequality of multisets $G\cup G'\ge[k]$, then $$\zcl(M(\bl))\ge \begin{cases}k+2&k\equiv m\ (2)\\
    k+1&k\not\equiv m\ (2).\end{cases}$$
\item[b.] If there are no exotic products in $H^*(M(\bl))$,  then (a) is sharp in the sense that if
$$k_0:=\max\{k: \exists\text{ gees $G$ and $G'$ of } \bl \text{ with }G\cup G'\ge[k]\},$$
then
$$\zcl(M(\bl))=\begin{cases}k_0+2&k_0\equiv m\ (2)\\
k_0+1&k_0\not\equiv m\ (2).\end{cases}$$
\end{enumerate}
\end{thm}
\ni The result in (b) says that zcl is the smallest integer $>k_0$ with the same parity as $m$.
\end{minipage}
\end{center}

\medskip
Note that (b) holds if $m\ge 2s$, where $s$ denotes the maximum size of the gees of $\bl$. In the example $M_{2k+1}$ mentioned at the end of Section \ref{EMSSsec},
we obtain $\zcl=2k$, hence $2k+1\le\TC(M_{2k+1})\le 4k-3$, so there is a big gap  here.
\begin{proof} Under the hypothesis of (a), there is a partition $[k]=S\squarecup T$ with $G\ge S$, and $G'\ge T$. Then the following product of $k+1$ zero-divisors is nonzero:
$$\prod_{i\in S}\ol{V_i}\cdot \ol{W_S}\cdot \prod_{j\in T}\ol{V_j}.$$
Indeed, this product contains the nonzero term $W_\emptyset\ot V_T$, and this term cannot be cancelled by any other term in the expansion, since the only way to obtain $W_\emptyset$ is as $V_UW_U$ for some set $U$. The stronger result when $k\equiv m\pmod2$ is obtained using
$$\prod_{i\in S}\ol{V_i}\cdot \ol{W_S}\cdot \prod_{j\in T}\ol{V_j}\cdot\ol{W_T},$$
which is nonzero by Remark \ref{rem}.

Part (a) implies $\ge$ in (b). We will prove $\le$ by showing
 that, under the assumption that there are no exotic products, if there is a nonzero product of $k+1$ zero divisors with $k\equiv m\pmod2$, then there are gees $G$ and $G'$ of $\ell$ such that $G\cup G'\ge[k]$. This says that if $\zcl\ge k+1$ with $k\equiv m\pmod2$, then $k_0\ge k$. Thus if $\zcl\ge\begin{cases}k_0+3&k_0\equiv m\\
 k_0+2&k_0\not\equiv m,\end{cases}$ then $k_0\ge k_0+2$ (resp., $k_0+1$), a contradiction.

 We begin by considering the case when all the zero divisors are of the form $\ol{V_i}$ or $\ol{W_S}$. Since products of $W$'s are 0, there cannot be more than two $\Wbar$'s. The case of no $\Wbar$'s is easiest and is omitted. Denote $\Vbar_S:=\prod_{i\in S}\ol{V_i}$. Note the distinction: $\ol{W_S}=W_S\ot1+1\ot W_S$, whereas $\Vbar_S=\prod_{i\in S}(V_i\ot1+1\ot V_i)$,
 with the usual convention that the entries of $S$ are listed in increasing order.

For the case of one $\Wbar$, assume $\Vbar_{T_1}\Vbar_{T_2}\ol{W_S}\ne0$ with $T_1\subset S$, $T_2$ and $S$ disjoint, and $|T_1\cup T_2|\ge k$.
Since $W_S\ne0$, $S\subset G$ for some gee $G$. The product expands, up to $\pm$ signs on terms, as
$$\sum_{T'\subset T_1} V_{T_2}V_{T'}\ot W_{S-T'}+W_{S-T'}\ot V_{T_2}V_{T'}.$$
For this to be nonzero, we must have $V_{T_2}\ne0$, and so $T_2\le G'$ for some gee $G'$. Thus
$$[k]\le T_1\cup T_2\le G\cup G'.$$

For the case of two $\Wbar$'s, we may assume that
\begin{equation}\label{prd}\Vbar_{E_1}\Vbar_{E_2}\Vbar_{E_3}\ol{W_{D_1\cup D_3}}\ \ol{W_{D_2\cup D_3}}\ne0,\end{equation}
with $D_1$, $D_2$, and $D_3$ disjoint,  $E_i\subset D_i$, and $|E_1\cup E_2\cup E_3|=k-1$ with $k\equiv m$ mod 2.
Note that we cannot have a factor $\Vbar_{E_4}$ with $E_4$ disjoint from $D_1\cup D_2\cup D_3$, since
$$\ol{W_{D_1\cup D_3}}\ \ol{W_{D_2\cup D_3}}=-W_{D_1\cup D_3}\ot W_{D_2\cup D_3}\pm W_{D_2\cup D_3}\ot W_{D_1\cup D_3},$$
and the product of $\Vbar_{E_4}$ with this would be 0.

Since $W_{D_i\cup D_3}\ne0$ for $i=1,2$, we have
$$E_i\cup E_3\subset D_i\cup D_3\le G_i$$
for gees $G_i$. Thus
$$G_1\cup G_2\ge D_1\cup D_2\cup D_3\supset E_1\cup E_2\cup E_3\ge
[k-1],$$
and so $G_1\cup G_2\ge [k]$ unless  each $D_i=E_i$, $1\le i\le3$. But
in this case the LHS of (\ref{prd}) is 0 by Lemma \ref{ratlem}, since $|D_1\cup D_2\cup D_3|\not\equiv m\pmod2$ in this case. This completes the proof when all zero divisors are of the form $\ol{V_i}$ or $\ol{W_S}$.

Let $R=H^*(M(\ell))$. In Lemma \ref{zd}, we show that any product $P$ of $z$ zero divisors can be written as $\sum \a_iP_i$, where $\a_i\in R\ot R$ and $P_i$ is a product of $z$ factors of the form $\Vbar_i$ or $\ol{W_S}$. If $P\ne0$, then some $P_i$ must be nonzero, and so by the above argument there exist gees $G$ and $G'$ as claimed.
\end{proof}

The following two lemmas were used in the preceding proof.
\begin{lem}\label{ratlem} If there are no exotic products in $H^*(M(\ell))$, and $D_1$, $D_2$, and $D_3$ are pairwise disjoint subgees, then
$$\Vbar_{D_1}\Vbar_{D_2}\Vbar_{D_3}\ol{W_{D_1\cup D_3}}\,\ol{W_{D_2\cup D_3}}\ne0\text{ iff }|D_1\cup D_2\cup D_3|\equiv m\pmod2.$$
\end{lem}
\begin{proof} The notation
$$\la S,T,e\ra:=W_S\ot W_T+(-1)^e W_T\ot W_S,$$
with $S$ and $T$ disjoint,
will be useful in this proof. We begin with the observation that
if $k\not\in S\cup T$, then
\begin{equation}\label{sublem}\ol{V_k}\la S\cup\{k\},T,e\ra=\pm\la S,T,e+|W_T|+1\ra.\end{equation}
Indeed, since $|V_k|=1$, we have
\begin{eqnarray*}&&(V_k\ot1-1\ot V_k)(W_{S\cup\{k\}}\ot W_T+(-1)^eW_T\ot W_{S\cup\{k\}})\\
&=&V_kW_{S\cup\{k\}}\ot W_T-(-1)^{|W_T|}(-1)^eW_T\ot V_kW_{S\cup\{k\}}.\end{eqnarray*}
The $\pm$, which is not important, is $(-1)^{\rho_k(S)}$ from Proposition \ref{VW}.
Similarly,
\begin{equation}\label{36} \Vbar_k\la S,T\cup\{k\},e\ra=\pm\la S,T,e+|W_S|+1\ra.\end{equation}

Since products of $W$'s are  0 by assumption, we have
$$\ol{W_{D_1\cup D_3}}\,\ol{W_{D_2\cup D_3}}=-\la D_1\cup D_3,D_2\cup D_3, |W_{D_1\cup D_3}|\cdot|W_{D_2\cup D_3}|\ra.$$
Now apply (\ref{36}) $|D_2|$ times to this to eliminate the elements of $D_2$, each time adding $|W_{D_1\cup D_3}|+1$ to the third component of $\la-,-,-\ra$, obtaining
$$\Vbar_{D_2}\ol{W_{D_1\cup D_3}}\,\ol{W_{D_2\cup D_3}}=\pm\la D_1\cup D_3,D_3,|W_{D_1\cup D_3}|\cdot|W_{D_2\cup D_3}|+|D_2|(|W_{D_1\cup D_3}|+1)\ra.$$
Let $d_i=|D_i|$ and note that $|W_{D_i}|=m-d_i$.
Now apply (\ref{sublem}) $d_1$ times, eliminating the elements of $D_1$. We obtain that the expression in the lemma equals
\begin{equation}\pm\Vbar_{D_3}\ \la D_3,D_3,f\ra\label{ex}\end{equation}
with $$f=(m-d_1-d_3)(m-d_2-d_3)+d_2(m-d_1-d_3+1)+d_1(m-d_3+1)\equiv m+d_1+d_2+d_3\pmod2.$$
Thus (\ref{ex}) equals 0 if $m+d_1+d_2+d_3$ is odd, while if $m+d_1+d_2+d_3$ is even, it equals
$$\pm2\Vbar_{D_3}(W_{D_3}\ot W_{D_3})=\pm2(W_\phi\ot W_{D_3}+\text{other terms})\ne0.$$
\end{proof}

\begin{rmk}\label{rem} {\rm Note that the backwards implication in Lemma \ref{ratlem} is true without the assumption of no exotic products because the only additional terms will involve just products of $V$'s, and these cannot cancel $W_\phi\ot W_{D_1}$.}\end{rmk}

\begin{lem}\label{zd} If $R=H^*(M(\ell))$ has no exotic products, then every zero divisor of $R\ot R$ is in the ideal spanned by elements of the form $\Vbar_i$ and $\ol{W_S}$.
\end{lem}
\begin{proof} The vector space $R\ot R$ is spanned by monomials of three types: (1) $W_S\ot W_T$; (2) $V_S\ot V_T$; and (3) $V_S\ot W_T$ and $W_T\ot V_S$. If a zero divisor is written as $Z_1+Z_2+Z_3$, where $Z_i$ is of type $i$, then each $Z_i$ must be a zero divisor, since their images under multiplication $m$ are, respectively 0, $V$'s, and $W$'s. We show that each type of zero divisor is in the claimed ideal.

(1) Every monomial $W_S\ot W_T$ is a zero divisor and can be written as $-(W_S\ot1)(W_T\ot 1-1\ot W_T)$.

(2) There are three types of zero divisors of this type. (a) One of the form $V_iV_S\ot V_iV_T$ equals $\pm(V_iV_S\ot V_T)(V_i\ot1-1\ot V_i)$. (b) If $V_SV_T=0$ in $R$, then $V_S\ot V_T$ is a zero divisor and equals $-(V_S\ot 1)(V_T\ot1-1\ot V_T)$, and one easily shows $V_T\ot1-1\ot V_T$ is in the ideal by induction on $|T|$. (c) If $V_S\ne0$ and $S=S_i\sqcup T_i$ so that $V_{S_i}V_{T_i}=(-1)^{\eps_i}V_S$, then $\sum c_iV_{S_i}\ot V_{T_i}$ is a zero divisor if $\sum(-1)^{\eps_i}c_i=0$. We prove by induction on $|S|$ that these zero divisors have the required form. WLOG, assume that $1\in S$. For every $i$ with $1\in T_i$, let $\widetilde T_i=T_i-\{1\}$, and write
$$V_{S_i}\ot V_{T_i}=\pm(V_1\ot1-1\ot V_1)V_{S_i}\ot V_{\widetilde T_i}\pm V_1V_{S_i}\ot V_{\widetilde T_i}.$$
In this way, the given zero divisor can be written as a sum of terms of the desired form plus a sum of terms with $V_1$ in the left factor of each. These latter terms can be written as $V_1\ot1$ times a sum with smaller $|S|$, and this can be written in the desired form by the induction hypothesis.

(3) There are zero divisors of the form
$$\sum_i c_iV_{S_i}\ot W_{T\cup S_i}+\sum_j d_jW_{T\cup S_j}\ot V_{S_j},$$
with $S_i$ and $S_j$ disjoint from $T$, $V_{S_i}W_{T\cup S_i}=(-1)^{\eps_i}W_T$, $W_{T\cup S_j}V_{S_j}=(-1)^{\eps_j}W_T$, and $\sum(-1)^{\eps_i}c_i+\sum(-1)^{\eps_j}d_j=0$. We claim that each term $V_{S_i}\ot W_{T\cup S_i}$ is equivalent, mod terms of the desired form, to $1\ot W_T$, and similarly for $W_{T\cup S_j}\ot V_{S_j}$. Thus the given zero divisor is equivalent, mod things of the desired form, to a multiple of $W_T\ot1-1\ot W_T$.

The claim is proved by induction on $|V_{S_i}|$, noting that if $s$ is the smallest element of $S_i$, then
$$V_{S_i}\ot W_{T\cup S_i}=(V_s\ot1-1\ot V_s)(V_{S_i-\{s\}}\ot W_{T\cup S_i})\pm V_{S_i-\{s\}}\ot W_{T\cup S_i\cup\{s\}}.$$
\end{proof}

 Our zcl results depend only on the gees and the parity of $n$, and not on the value of $n$.  (Recall $m=n-3$.) However the possible gees depend on $n$.  Of course, the numbers which occur in the gees must be less than $n$, but also, if $G$ and $G'$ are gees (not necessarily distinct), then we cannot have $[n-1]-G'\le G\cup\{n\}$, for then $G\cup\{n\}$ would be both short and  long. Thus, for example, $8531$ is an allowable gene, but $7531 $ is not, since $642<7531$ but $7642\not<8531$.

There are 2469 equivalence classes of nonempty  spaces $M(\bl)$ with $n=8$. Genes for these are listed in \cite{web}. We perform an analysis of what we can say about the zcl and TC of these.
Since $n=8$, each satisfies $\TC(M(\bl))\le 11$ by (\ref{TCM}). As we discuss below in more detail, for most of them we can assert that $\zcl(M(\bl))\ge7$, and so $\TC(M(\bl))\ge8$. For most of them we can only assert lower bounds for zcl, due to the possibility of exotic products. We emphasize that the following analysis pertains to the case $n=8$.

As discussed in Proposition \ref{notmax} and the paragraph which preceded it, there is only one $\bl$ with a gee of size 5. This $M(\bl)$ is homeomorphic to $T^5\sqcup T^5$ with topological complexity 6. This is a truly special case, as it is the only disconnected $M(\bl)$.

Other special cases which we wish to exclude from the analysis below are those in Table \ref{cases} below, which are completely understood by elementary means. Sides of ``length 0'' stand for sides of very small length.
The identification as spaces is from \cite[Prop 2.1]{Geom} and \cite[Expl 6.5]{H}, with $T^k$ a $k$-torus. The upper bound for TC follows from \cite[Prop 4.41 and Thm 4.49]{F}, and the lower bound from Theorem \ref{lower}(a) and (\ref{zclz}).
We thank the referee for suggesting this table.

\begin{center}
\begin{minipage}{6.5in}
\begin{tab}{Special cases}

\label{cases}

\begin{center}
\renewcommand\arraystretch{1.2}
\begin{tabular}{ccccc}

$\ell$&Gen.~code&space&$\zcl$&$\TC$\\
\hline
$(1,1,1,1,1,1,1,6)$&$\la8\ra$&$S^5$&$1$&$2$\\
$(0,1,1,1,1,1,1,5)$&$\la81\ra$&$S^4\times T^1$&$3$&$4$\\
$(0,0,1,1,1,1,1,4)$&$\la821\ra$&$S^3\times T^2$&$3$&$4$\\
$(0,0,0,1,1,1,1,3)$&$\la8321\ra$&$S^2\times T^3$&$5$&$6$\\
$(0,0,0,0,1,1,1,2)$&$\la84321\ra$&$T^5$&$5$&$6$

\end{tabular}

\end{center}

\end{tab}
\end{minipage}
\end{center}

\medskip
There are 768 $\bl$'s whose largest gee has size 4. For all of them, we can deduce only $\zcl(M(\bl))\ge7$, using Theorem \ref{lower}(a) and the following result.
\begin{prop}\label{propless} Suppose $G$ and $G'$ are subsets of $[7]$,  not necessarily distinct, with neither strictly less than the other and with $\max(|G|,|G'|)=4$. Assume also that it is not the case
that $G=G'=4321$, and it is not the case that $[7]-G'\le G\cup\{8\}$. Then $G\cup G'\ge[5]$ but $G\cup G'\not\ge[6]$.\end{prop}
\begin{proof} The first conclusion follows easily from the observation that if $G=4321$, then $5\in G'$. For the second, if $G\cup G'\ge[6]$ then applying $\cup G'$ to the false statement $[7]-G'\le G\cup\{8\}$ would yield a true statement,
and the ordering that we are using for multisets has a cancellation property for unions.\end{proof}

There are 1569 $\bl$'s whose largest gee has size 3.  By Theorem \ref{upper}, these all satisfy $\zcl\le8$. For these, we again cannot rule out exotic products, so we cannot use Theorem \ref{lower}(b) to get sharp zcl results.
Of these, 929 have a gee $G\ge531$, and this satisfies $G\cup G\ge[5]$, hence $\zcl\ge7$. In addition to these, there are 524 with distinct gees satisfying $G\cup G'\ge[5]$.
Combining these with the 768 $\bl$'s with some $|\text{gee}|=4$, we find that 2221 of the 2469
$\bl$'s with $n=8$ satisfy $\zcl(M(\bl))\ge7$. There are another 116 $\bl$'s with largest gee of size 3 for which we can only assert $\zcl\ge5$. An example of a genetic code of this type is $\la8421,843,862,871\ra$.

There are 120 $\bl$'s whose largest gee has size 2. For these, exotic products are not possible and we can assert the precise value of zcl. Of these, 85 have a gee $G\ge42$ and since $G\cup G\ge[4]$, they have $\zcl(M(\bl))=5$. In addition to these, there are 10 having distinct gees satisfying $G\cup G'\ge[4]$ and so again zcl$=5$. There are 25 others for which we only have $G\cup G'\ge[3]$, but still zcl$=5$.
Finally, there are 6 $\bl$'s with largest gee of size 1. These satisfy $\zcl(M(\bl))=3$.

In Table \ref{T1}, we summarize what we can say about zcl when $n=8$, omitting the six special cases described earlier. Keep in mind that
$$1+\zcl\le\TC\le11.$$
In the table, $s$ denotes the size of the largest gee, and $\#$ denotes the number of distinct homeomorphism classes of 8-gons having the property.

\begin{minipage}{6.5in}
\begin{center}
\begin{tab}{Number of types of 8-gon spaces}

\label{T1}

\begin{center}

\begin{tabular}{ccr}
$s$&$\zcl$&$\#$\\
\hline
$1$&$3$&$6$\\
$2$&$5$&$120$\\
$3$&$5,6,7$ {\rm or }$8$&$116$\\
$3$&$7$ {\rm or }$8$&$1453$\\
$4$&$7,8,9$ {\rm or }$10$&$768$
\end{tabular}

\end{center}

\end{tab}
\end{center}
\end{minipage}

For general $m$($=n-3$), the largest gees (with one exception) have size $s=m-1$, and so Theorem \ref{upper} allows the {\it possibility} of zcl as large as $2m$, which would imply $\TC=2m+1$ by (\ref{TCM}). However, this would require many nontrivial exotic products. By an argument similar to Proposition \ref{propless}, all we can assert from Theorem \ref{lower}(a) is $\zcl\ge m+2$ (when $s=m-1$).
If $s\le[m/2]$, then we can determine the precise zcl, which can be as large as $2s+2$, so we can obtain $m+1$ or $m+2$ as zcl,  yielding a lower bound for TC only roughly half the upper bound given by (\ref{TCM}).

\section{An example with a nontrivial exotic product}\label{sec4}
 The following example, provided by the referee, suggests that additional geometric information may be needed in finding exotic products and sharper zcl bounds.

\begin{thm} Let $X=M(\ell)$ with genetic code $\la 632\ra$. There are exotic products in $H^*(X;\Q)$. The  (rational) zcl of $X$ is 6, and $\TC(X)=7$.
\end{thm}
\begin{proof} The space $X$ is homeomorphic to the connected sum of two 3-tori by \cite[(2) in Expl 2.11]{Geom}. The length vector $\ell$ could be taken to be $(1,1,1,3,3,4)$, although this is irrelevant to the proof. An elementary argument is presented in \cite[Prop 4.2.1]{book} that there is a ring isomorphism in positive dimensions $H^*(X)\approx (H^*(T^3)\oplus H^*(T^3))/(a_1a_2a_3-b_1b_2b_3)$ with any coefficients. Here $a_i$ and $b_i$ are the generators of the first cohomology groups of the two 3-tori.
We have $a_i^2=0=b_i^2$ and $a_ib_j=0$.

If there were no exotic products in its rational cohomology, then, by  Theorem \ref{lower}(b), since $m=3$ and $k_0=3$, $\zcl(X)$ would equal 5. However, in $H^*(X\times X;\Q)$, we have
$$\abar_1\abar_2\abar_3\bbar_1\bbar_2\bbar_3=a_1a_2a_3\ot b_1b_2b_3+b_1b_2b_3\ot a_1a_2a_3.$$
Since $a_1a_2a_3=b_1b_2b_3$, this  equals 2 times the top class of $H^*(X\times X;\Q)$. Thus the rational zcl equals 6, and $\TC(X)=7$ by (\ref{TCM}) and (\ref{zclz}).

An isomorphism between the $(a,b)$- and $(V,W)$-presentations is given by $V_i=a_i+b_i$,  $W_{1,2}=a_3$, $W_{1,3}=b_2$, $W_{2,3}=a_1$, $W_1=b_2b_3$, $W_2=a_1a_3$, and $W_3=b_1b_2$.
One can check that this satisfies Theorem \ref{basis}, but has exotic products such as $W_{1,2}W_{2,3}=-W_2$ and $V_2W_{1,2}=-W_1+V_2V_3$.
\end{proof}
 \def\line{\rule{.6in}{.6pt}}

\end{document}